\documentclass{amsart}
\usepackage{color}

\newtheorem{theorem}{Theorem}%[section]
\newtheorem{lemma}[theorem]{Lemma}
\newtheorem{proposition}[theorem]{Proposition}
\newtheorem{corollary}[theorem]{Corollary}

\theoremstyle{definition}

\newtheorem{example}[theorem]{Example}

\theoremstyle{remark}
\newtheorem{remark}[theorem]{Remark}

\numberwithin{equation}{section}

\newcommand{\R}{\mathbb{R}}
\newcommand{\C}{\mathbb{C}}

\newcommand{\ov}{\overline}
\newcommand{\begeq}{\begin{equation}}
\newcommand{\stopeq}{\end{equation}}

\begin{document}

\title[Hypocomplex vector fields]
{Properties of solutions of a class of hypocomplex vector fields}

\author{C. Campana}
\address{\small Departamento de Matem\'atica, Instituto de Ci\^encias Matem\'aticas e de\break Computa\c c\~ao,
Universidade de S\~ao Paulo, Caixa Postal 668, S\~ao Carlos, SP 13560-970, Brazil}
\email{camilo@icmc.usp.br}
\thanks{The first author was supported by FAPESP}

\author{P. L. Dattori da Silva}
\address{\small Departamento de Matem\'atica, Instituto de Ci\^encias Matem\'aticas e de\break Computa\c c\~ao,
Universidade de S\~ao Paulo, Caixa Postal 668, S\~ao Carlos, SP 13560-970, Brazil}
\email{dattori@icmc.usp.br}
\thanks{The second author was supported in part by CNPq and FAPESP}

\author{A. Meziani}
\address{\small Department of Mathematics, Florida International
University, Miami, FL, 33199, USA}
\email{meziani@fiu.edu}

%\subjclass[2000]{Primary 35A01; Secondary 30G20, 35J70}

%\keywords{Global solvability, degenerate elliptic, periodic solutions, vector fields}

\begin{abstract}A Cauchy type integral operator is associated to a class of
integrable vector fields with complex coefficients. Properties of the integral
operator are used to deduce H\"{o}lder solvability of semilinear equations
$Lu=F(x,y,u)$ and a strong similarity principle between the solutions
of the equation $Lu=au+b\overline{u}$ and those of the equation $Lu=0$.
\end{abstract}
\maketitle

\section{Introduction}

This paper explores the extent to which properties of the Cauchy-Riemann operator $\partial/\partial\overline{z}$ extend to planar complex vector fields.
The class of vector fields amenable to possess such properties is within the so called ``hypocomplex" vector fields. In the plane these vector fields have first
integrals that are local homeomorphisms. Any such vector field $L$ with $C^\infty$ coefficients is solvable in the $C^\infty$ category: if $f\in C^\infty(\mathcal{U})$
then there exists $u\in C^\infty(\mathcal{U})$ satisfying the equation
\begin{equation}\label{eq_1_intro}
Lu=f.
\end{equation}
In fact, $L$ is also hypoelliptic. It is also proved (see \cite{BeCH} or \cite{Hounie_Perdigao}) that if $f\in L^p$
then  (\ref{eq_1_intro}) {has solutions} in $L^p$.
However, in general the solution $u$ cannot be expected to be more regular than the right hand side. Indeed, in \cite{BeCH} the authors give an example of a
$C^\infty$ hypocomplex vector field $L$ and a function $f\in L^\infty(\mathbb{R}^2)$ such that (\ref{eq_1_intro}) does not have {$L^\infty$ solutions} in
any neighborhood of $0\in\R^2$.

In this paper we isolate a class of hypocomplex vector fields, with properties analogous to those of $\partial/\partial\overline{z}$.
This class consist of those vector
fields $L$, defined in a bounded domain $\Omega$, that are locally equivalent to a multiple of the vector field
\begin{equation}\label{eq_2_intro}
\partial/\partial y -i|y|^{\sigma_\mathsf{p}}\partial/\partial x, \quad \sigma_\mathsf{p}>0,
\end{equation}
in a neighborhood of each point $\mathsf{p}$ where $L$ fails to be elliptic. In this case, we prove that if $f\in L^p(\Omega)$, with $p>2+\sigma$ and
$\sigma=\max\{\sigma_\mathsf{p}\}$, then all solutions $u$ of (\ref{eq_1_intro}) are in $C^\alpha(\Omega)$, where $\alpha=(2-q-\tau)/q$, with
$\tau=\sigma/(\sigma+1)$ and $q=p/(p-1)$. This result is obtained through the study of the integral operator
\begin{equation}\label{eq_3_intro}
T_Zf(x,y)=\frac{1}{2\pi i}\int_\Omega\frac{f(\xi,\eta)}{Z(\xi,\eta)-Z(x,y)}d\xi d\eta,
\end{equation}
where $Z:\Omega\rightarrow\mathbb{C}$ is a global first integral of $L$. This result is then applied to show that the {semilinear} equation
\begin{equation}\label{eq_4_intro}
Lu=F(x,y,u)
\end{equation}
has H\"older continuous solution for a class of functions $F$. As a consequence we obtain a strong similarity principle for the solutions of the equation
\begin{equation}\label{eq_5_intro}
Lu=Au+B\overline{u},
\end{equation}
with $A,B\in L^p(\Omega)$. That is, any solution of (\ref{eq_5_intro}) is in $C^\alpha(\Omega)$ and satisfies $u=H(Z)e^s$, where $H$ is a holomorphic function
defined in $Z(\Omega)$ and $s\in C^\alpha$.

It should be noted that analogous questions were investigated in \cite{Mez-CM05} for vector field (\ref{eq_2_intro}) with $\sigma\in 2\mathbb{Z}_+$ and where $f\in L^\infty$
(resp, $A,B\in L^\infty$).
The approach and motivation for this paper are related to \cite{BeCH}, \cite{Ber-Hou-San}, \cite{Begehr}, \cite{Bers},
 \cite{Mez-JDE99}, \cite{Mez-Mem}, \cite{Mez-CV08}, \cite{Mez-CM05}, \cite{Vekua}, and many others.
This paper is organized as follows. After the necessary preliminaries of section 2 and technical Lemmas of section 3, we study properties of the integral operator
$T_Z$ in section 4 and 5. In particular, we prove that $T_Zf$ solves (\ref{eq_1_intro}) for any $f\in L^1(\Omega)$,  $T_Zf\in L^q(\Omega)$ for any
$1\leq q<2-\tau$, and $T_Zf\in C^{(2-q-\tau)/q}(\Omega)$ if $f\in L^p(\Omega)$ with $p>2+\sigma$. In section 6 we study the {semilinear} equation and deduce the
similarity principle.

This work was done when the first and second authors were visiting the Department of Mathematics \& Statistics at Florida International University.
They are grateful and would like to thank the members of the host institution for the support they provided during the visit.

%%%%%%%%%%%%%%%%%%% preliminaries %%%%%%%%%%%%%%%%

\section{Preliminaries}\label{preliminaries}

Let
\[
L=A(x,y)\partial/\partial x+B(x,y)\partial/\partial y
\]
be a complex vector field defined in a region $\tilde{\Omega}\subset\mathbb{R}^2$, where $A$ and $B$ are $\mathbb{C}$-valued H\"older continuous functions in $\tilde\Omega$,
$|A|+|B|>0$ in $\tilde\Omega$. Let $\Omega$ be an open set such that
$\overline{\Omega}\subset\tilde{\Omega}$. Hypocomplex vector fields were introduced (see \cite{BeCH} or \cite{Treves-hypo}) as those vector fields
that are locally integrable and such that any solution of $Lu=0$ can be written locally as $h\circ Z$ with $h$ holomorphic and $Z$ a first integral.
It turns out that in the case of vector fields in the plane, this is equivalent to requiring that any first integral of $L$ is a local homeomorphism.
In this paper, we consider
a vector field $L$ to be \textit{hypocomplex} on $\overline{\Omega}$ if for every $\mathsf{p}\in\overline{\Omega}$ there exists a $C^{1+\epsilon}$  function (for some $\epsilon >0$)
$Z:\mathcal{U}\rightarrow\mathbb{C}$ defined in an open set $\mathcal{U}$, $\mathsf{p}\in\mathcal U$, such that $dZ\neq0$, $LZ=0$, and
$Z:\mathcal{U}\rightarrow Z(\mathcal{U})$ is a homeomorphism.
The set where $L$ fails to be elliptic is given by
\[
\Sigma=\{\mathsf{p}\in\overline{\Omega}; \, L_\mathsf{p}\wedge\overline{L}_\mathsf{p}=0\}=\{\mathsf{p}\in\overline{\Omega}; \, Im(A\overline{B})(\mathsf{p})=0\},
\]
where $\overline{L}=\overline{A}(x,y)\partial/\partial x+\overline{B}(x,y)\partial/\partial y$ is the complex conjugate of $L$. We will refer to $\Sigma$
as the characteristic set of $L$ ($\Sigma$ is in fact the base projection of the characteristic set of the first order differential operator $L$).

The class of vector fields under study in this paper are those vector fields that satisfy the following conditions:

\begin{itemize}
\item[(i)] $L$ is hypocomplex in $\overline{\Omega}$
\item[(ii)] The characteristic set $\Sigma\subset \overline{\Omega}$ is a $C^{1+\epsilon}$ curve nontangent to $L$
\item[(iii)] For every $\mathsf{p}\in\Sigma$, there exists an open set $\mathcal U$, with $\mathsf{p}\in\mathcal U$, such that $\Sigma\cap\mathcal{U}$ is given by a defining
$C^{1+\epsilon}$ function $\rho(x,y)$ such that $\Im(A\overline{B})(x,y)=|\rho(x,y)|^\sigma g(x,y)$, for some continuous function $g$ in $\mathcal U$ satisfying $g(x,y)\neq0$ for all $(x,y)\in\mathcal{U}$.
\end{itemize}

It should be noted that these conditions are invariant under $C^{1+\epsilon}$ change of variables. The following Proposition gives a local normal form for vector fields satisfying
the above conditions.

\begin{proposition}\label{normalization}
Suppose that $L$ satisfies {\em(i)}, {\em(ii)}, and {\em(iii)}. Then, for every $\mathsf{p}\in\Sigma$, there exist
an open neighborhood $U$ such that
$U\backslash\Sigma$ consists of two connected components
$U^+$ and $U^-$, and local coordinates $(x^+,t^+)$ $($respectively $(x^-,t^-)$$)$ centered at $\mathsf{p}$ such that $L$ is a multiple
of the following vector field in $U^+$ $($respectively $U^-$$)$:
\begin{equation}\label{L_normalized}
L_\sigma=\partial/\partial t^\pm -i|t^\pm |^\sigma\partial/\partial x^\pm,
\end{equation}
with first integral
\begin{equation}\label{Z_normalized}
Z_\sigma(x,t)=x^\pm +i\frac{t^\pm |t^\pm |^\sigma}{\sigma+1}
\end{equation}
\end{proposition}

\begin{proof}
Since $L$ is hypocomplex and $\Sigma$ is smooth then we can assume that there is in a neighborhood of $\mathsf{p}\in\Sigma$, such that the set $\Sigma\cap\mathcal{U}$ is given by
$\{y=0\}$ and that the first integral has the form $Z(x,y)=x+i\varphi(x,y)$, with $\varphi$ real-valued (see \cite{BeCH}, \cite{Treves-hypo}).
Thus $L$ is a multiple of the Hamiltonian
\[
Z_x\partial/\partial y -Z_y\partial/\partial x
\]
and condition (iii) implies that
\[
\frac{\partial\varphi}{\partial y}(x,y)=|y|^\sigma\psi(x,y),
\]
with $\psi$ a $C^0$ function, and $\psi(0,0)\neq0$. %We can assume without loss of generality that $\psi(0,0)>0$.
Then,
\[
\varphi(x,y)=\int_0^y|s|^\sigma\psi(x,s)\,ds\, +\beta (x)=\tilde{\varphi}(x,y)+\beta(x),
\]
for some $C^{1+\epsilon}$-function $\beta$.
With respect to the new variables
\[
x=x, \quad t=\textrm{sgn}(\tilde{\varphi})|\tilde{\varphi}(x,y)(1+\sigma)|^\frac{1}{1+\sigma}
\]
the expression of $Z$ becomes
\[
Z(x,t)=x+i\left(\frac{t|t|^\sigma}{\sigma+1}+\beta(x)\right)\, .
\]

Let $D$ be a small disc centered at 0 such that $\Sigma $ divides
$D$ into two semi discs $D^+=\{ t>0\}$ and $D^-=\{t<0\}$.
Note that $Z(D^+)$ and $Z(D^-)$ are simply connected in $\C$ sharing a boundary
curve $\gamma =\{ (x,\beta (x))\}$. Since $\gamma$ is $C^{1+\epsilon}$-curve, then we can find conformal mappings
\[
H^\pm :\, Z(D^\pm)\, \longrightarrow\, H^\pm (Z(D^\pm))\subset\C
\]
sending the boundary curve $\gamma$ into the real axis and $H^\pm$ extends as a $C^1$-diffeomor\-phism to a full neighborhood of 0.
The function $Z^\pm (x,t) =H^\pm (Z(x,t))$ is then a first integral of $L$ in $D^\pm$ satisfying $\Im Z^\pm(x,0)=0$.
It follows that $\Im Z^\pm (x,t)=\displaystyle\frac{t|t|^\sigma}{\sigma+1} \tilde{\psi}(x,t)^{1+\sigma}$ for some function $\tilde{\psi}(x,t)$ (with
$\tilde{\psi}(0,0)>0$). With respect to the new coordinates
\[
x^\pm =\Re Z^\pm (x,t),\quad t^\pm =t\tilde{\psi}(x,t)
\]
the vector field $L$ becomes a multiple of the desired vector field $L_\sigma$ given in the Proposition.

\end{proof}

\begin{remark}
Note that if $L$ is $C^\infty$, locally integrable and, of constant type along each connected component of $\Sigma$, then it satisfies (iii) (see \cite{BeCH}, \cite{Treves-hypo}).
Indeed, in this case if $\Sigma_j$ is connected component of $\Sigma$ and $L$ is of constant type $n_j$ along $\Sigma_j$, then it can be shown (as in \cite{Mez-CM05}) that for each point $p\in\Sigma_j$
coordinates $(x^\pm,t^\pm)$ can found in which the expression of $L$ is as in the Proposition \ref{normalization}.
\end{remark}

\begin{remark}
The vector field (of infinite type)
\[
L=\frac{\partial}{\partial t} -i\frac{e^{-\frac{1}{|t|}}}{t^2}\frac{\partial}{\partial x},
\]
is of class $C^\infty$, with characteristic set $\Sigma=\{t=0\}$. Also, $L$ is hypocomplex with (global)
first integral
\[
Z(x,t)=x+i\frac{t}{|t|}e^{-\frac{1}{|t|}}.
\]
However, the condition (iii) is not satisfied, since $e^{-\frac{1}{|t|}}\neq O(|t|^\sigma)$, for any $\sigma>0$.
\end{remark}

We close this section with the following Proposition.

\begin{proposition}\label{first_integral}
Let $L$ be a hypocomplex vector field with local first integrals  defined in an open set $\tilde{\Omega}\subset\mathbb{R}^2$.  Then for any open set
$\Omega\subset\subset\tilde{\Omega}$,
$L$ has a global first integral on $\overline{\Omega}$. More precisely, there exists $C^{1+\epsilon}$ function
\[
Z:\, \ov{\Omega}\, \longrightarrow\, Z(\ov{\Omega})\subset\subset\mathbb{C}
\]
such that $Z$ is a homeomorphism, $LZ=0$ and $dZ\ne 0$.
\end{proposition}

\begin{proof}
Since $L$ is hypocomplex, then for every $\mathsf{p}\in\ov{\Omega}$ there exists a $C^{1+\epsilon}$-first integral
$Z_\mathsf{p}$ defined in an open set $U_\mathsf{p}\subset\tilde{\Omega}$ with $\mathsf{p}\in U_\mathsf{p}$.
The collection $\{ U_\mathsf{p},Z_\mathsf{p}\}_{\mathsf{p}\in\ov{\Omega}}$ defines a structure of a Riemann surface on the open
set $\bigcup_{\mathsf{p}\in\ov{\Omega}}U_\mathsf{p}$, since the transition functions $h_{\mathsf{p}\mathsf{q}}=Z_\mathsf{p}\circ Z_\mathsf{q}^{-1}$ are holomorphic
functions on $Z_\mathsf{q}(U_\mathsf{p}\cap U_\mathsf{q})$.
The existence of a global first integral $Z$ follows from the Uniformization Theorem of the planar Riemann surfaces
 $\bigcup_{\mathsf{p}\in\ov{\Omega}}U_\mathsf{p}$ (see {\cite{Springer}}).
\end{proof}

\begin{remark}
Note that hypocomplexity of $L$ also implies that if $v$ satisfies $Lv=0$ in an open set
$O\subset \Omega$, then $v=h\circ Z$, where $h$ is a holomorphic function defined in $Z(O)$.
\end{remark}

%%%%%%% end of introduction %%%%%%%%%%%%%%%%%

%%%%%%%%%%%%%%%% some lemmas %%%%%%%%%%%%%

\section{Some Lemmas}

We prove some technical Lemmas that will be used in the following sections.

\begin{lemma}[\cite{Mez-CM05}, Lemma 3.1]\label{lemma_Mez-CM05} Let $0<\delta<R$, $0<\tau<1$, $m >0$, and $0\leq\gamma<R$. Then, there exists a constant $C(\tau)>0$ such that
\[
\int_\delta^R\!\!\int_0^{2\pi}\frac{d\theta dr}{|\gamma+r\sin\theta|^\tau r^{m+1}}\leq\frac{C(\tau)}{\delta^{m+\tau}}.
\]
\end{lemma}

\begin{lemma}\label{lemmaintegral}
Let $R>0$, $0<\tau<1$, $\gamma\in\mathbb{R}$, and $1 < q < 2-\tau$. Then, there exists a constant $M(q,\tau)>0$ such that
\[
I= \int_{0}^{2\pi}\int_{0}^{R}\frac{drd\theta}{|\gamma+r\sin\theta|^{\tau}r^{q-1}}\leq M(q,\tau)R^{2-\tau-q}.
\]
\end{lemma}
\begin{proof}
To prove the Lemma it is enough consider $\gamma\geq0$.
If $\gamma=0$ then
\[
I=\int_{0}^{2\pi}\int_{0}^{R}\frac{drd\theta}{|r\sin\theta|^{\tau}r^{q-1}}=
\int_{0}^{2\pi}\frac{d\theta}{|\sin\theta|^{\tau}} \int_{0}^{R}\frac{dr}{r^{\tau+q-1}}=  C(\tau)\displaystyle\frac{R^{2-\tau-q}}{2-\tau-q},
\]
where
\[
C(\tau)=\int_{0}^{2\pi}\frac{d\theta}{|\sin\theta|^{\tau}} <\infty.
\]

Now, suppose $\gamma>0$. Let $r=\rho \gamma$. We have
\begin{equation}\label{eq_J_1}
I=\gamma^{2-\tau-q}\int_{0}^{2\pi}\int_{0}^{\frac{R}{\gamma}}\frac{d\rho d\theta}{|\rho \sin\theta+1|^{\tau}\rho^{q-1}}\doteq \gamma^{2-\tau-q}\,J.
\end{equation}

To estimate $J$, we consider two cases: $0<\gamma<R$ and $\gamma\geq R$.\vspace{10pt}

Assume that $0<\gamma<R$. We can write
\[
J=
\int_{0}^{\pi}\int_{0}^{\frac{R}{\gamma}}\frac{d\rho d\theta}{|\rho \sin\theta+1|^{\tau}\rho^{q-1}}+\int_{\pi}^{2\pi}\int_{0}^{\frac{R}{\gamma}}\frac{d\rho d\theta}{|\rho \sin\theta+1|^{\tau}\rho^{q-1}}
\doteq J_1+J_2.
\]
For $\theta \in [0,\pi]$ we have $\sin\theta\geq 0$; consequently,  $\rho\sin\theta+1 \geq \rho\sin\theta$
 and $\rho\in [0,R/\gamma]$. Hence,
\[
J_1
\leq \int_{0}^{\pi}\int_{0}^{\frac{R}{\gamma}}\frac{d\rho d\theta}{|\rho \sin\theta|^{\tau}\rho^{q-1}}
\leq \frac{C(\tau)}{2-\tau-q}\left(\frac{R}{\gamma}\right)^{2-\tau-q} .
\]
Next, we will estimate $J_2$.
Let $\pi < \theta_{0} < 3\pi/2$ such that $-\sin\theta_{0} =\gamma/2R$. We can write
\begin{align*}
J_2 & =\int_{\pi}^{\theta_0}\int_{0}^{\frac{R}{\gamma}}\frac{d\rho d\theta}{|\rho \sin\theta+1|^{\tau}\rho^{q-1}}
+\int_{\theta_0}^{\frac{3\pi}{2}}\int_{0}^{\frac{R}{\gamma}}\frac{d\rho d\theta}{|\rho \sin\theta+1|^{\tau}\rho^{q-1}}\\
&
+\int_{\frac{3\pi}{2}}^{3\pi-\theta_0}\int_{0}^{\frac{R}{\gamma}}\frac{d\rho d\theta}{|\rho \sin\theta+1|^{\tau}\rho^{q-1}}
+\int_{3\pi-\theta_0}^{2\pi}\int_{0}^{\frac{R}{\gamma}}\frac{d\rho d\theta}{|\rho \sin\theta+1|^{\tau}\rho^{q-1}}\\
&
\doteq J_{2,1}+J_{2,2}+J_{2,3}+J_{2,4}.
\quad\quad\quad\quad\quad\quad\quad\quad\quad\quad\quad\quad\quad\quad\quad\quad\,\,
\end{align*}
Let $\varphi = -\theta + 3\pi$. We have
\[
J_{2,3}=\int_{\frac{3\pi}{2}}^{3\pi -\theta_{0}}\int_{0}^{\frac{R}{\gamma}}\frac{d\rho d\theta}{|\rho \sin\theta+1|^{\tau}\rho^{q-1}}
=\int_{\theta_{0}}^{\frac{3\pi}{2}}\int_{0}^{\frac{R}{\gamma}}\frac{d\rho d\varphi}{|\rho \sin\varphi+1|^{\tau}\rho^{q-1}}=J_{2,2}.
\]
Also,
\[
J_{2,4}=\int_{3\pi -\theta_{0}}^{2\pi}\int_{0}^{\frac{R}{\gamma}}\frac{d\rho d\theta}{|\rho \sin\theta+1|^{\tau}\rho^{q-1}}
= \int_{\pi}^{\theta_{0}}\int_{0}^{\frac{R}{\gamma}}\frac{d\rho d\varphi}{|\rho \sin\varphi+1|^{\tau}\rho^{q-1}}=J_{2,1}.
\]
%\vspace{10pt}
%\noindent\textit{Estimating $J_{2,1}$.}%\vspace{10pt}

Let us estimate $J_{2,1}$. Note that
\[
0\leq-\sin\theta\leq\frac{\gamma}{2R}<\frac{1}{2}, \quad \textrm{for} \quad 0\in[\pi,\theta_0];
\]
consequently, for $\theta\in [\pi,\theta_0]$ and $0<\rho<R/\gamma$, we have
\[
\rho|\sin\theta|<\frac{1}{2}<1+\rho\sin\theta.
\]
Hence,
\[
J_{2,1}
\leq \int_{\pi}^{\theta_{0}}\int_{0}^{\frac{R}{\gamma}} \frac{d\rho d\theta}{\rho^{\tau} |\sin\theta|^{\tau}\rho^{q-1}}
\leq \frac{C(\tau)}{4(2-\tau-q)}\left(\frac{R}{\gamma}\right)^{2-\tau-q}.
\]
\vspace{5pt}

Let us estimate $J_{2,2}$.
Let  $\varphi= \theta$ and $t=\rho \sin\theta+1$.
We have
\[
J_{2,2} =\int_{\theta_0}^{\frac{3\pi}{2}}\int_{0}^{\frac{R}{\gamma}}\frac{d\rho d\theta}{|\rho \sin\theta+1|^{\tau}\rho^{q-1}}
=\int_{\theta_{0}}^{\frac{3\pi}{2}}\int_{1}^{\frac{R}{\gamma} \sin\varphi+1}\frac{1}{\sin\varphi}\frac{dt d\varphi}{|t|^{\tau} \Big(\frac{t-1}{\sin\varphi}\Big)^{q-1}}
\]
\[
=\int_{\theta_{0}}^{\frac{3\pi}{2}} |\sin\varphi|^{q-2}\Big(\int_{\frac{R}{\gamma} \sin\varphi+1}^{1}\frac{dt}{|t|^{\tau}|t-1|^{q-1}}\Big)d\varphi
\quad\quad\quad\quad\quad\quad\quad\quad\quad\,
\]
\[
\leq\int_{\theta_{0}}^{\frac{3\pi}{2}} |\sin\varphi|^{q-2}\Big(\int_{-\frac{R}{\gamma}}^{1}\frac{dt}{|t|^{\tau}|t-1|^{q-1}}\Big)d\varphi.
\leq C(q,\tau)\left(\frac{R}{\gamma}\right)^{2-\tau-q},
\]
for some constant $C(q,\tau)>0$.

Hence, by the calculations above, we can find a constant $C_1(q,\tau)>0$ such that
\begin{equation}\label{eq_J_2}
J\leq C_1(q,\tau)\left(\frac{R}{\gamma}\right)^{2-\tau-q}.
\end{equation}
\vspace{5pt}

Now, assume that $\gamma\geq R$. In this case, there exists a constant $C(q,\tau)>0$ such that
\begin{equation}\label{eq_J_3}
J\!=\!\!\int_{0}^{2\pi}\!\!\!\int_{0}^{\frac{R}{\gamma}}\frac{d\rho d\theta}{|\rho\sin\theta+1|^{\tau}\rho^{q-1}}
\leq\int_{0}^{2\pi}\!\!\!\int_{0}^{\frac{R}{\gamma}}\frac{d\rho d\theta}{|1-\rho|^{\tau}\rho^{q-1}}
\leq C(q,\tau)\left(\frac{R}{\gamma}\right)^{2-\tau-q}
\end{equation}
%\vspace{10pt}

Finally, estimates (\ref{eq_J_1}), (\ref{eq_J_2}) and (\ref{eq_J_3}) show that
there exists a constant $M(q,\tau)>0$ such that
\[
I\leq M(q,\tau)R^{2-\tau-q}.
\]
\end{proof}

%%%%%%%%%%%%

\begin{lemma}\label{lemma_holder} Let $R>0$, $\gamma\geq0$, $0<\tau<1$, $1\leq q<2-\tau$, and $0\leq\varphi\leq\pi$. Then, there exists a constant $C(q,\tau)>0$ such that
\begin{equation}\label{H}
H=\int_0^{2\pi}\!\!\int_0^R\frac{r dr d\theta}{|\gamma+r\sin(\theta+\varphi)|^\tau r^q|re^{i\theta}-1|^q}\leq C(q,\tau).
\end{equation}
\end{lemma}
\begin{proof}
We will divide the proof in two cases: $\gamma=0$ and $\gamma>0$.
\vspace{10pt}

\noindent\textbf{Case 1}: $\gamma=0$.\ In this case we can write
\[
H=\int_{0}^{2\pi}\int_{0}^R\frac{dr d\theta}{|\sin(\theta+\varphi)|^{\tau}r^{q+\tau-1}|re^{i\theta}-1|^{q}}= H_1+H_2+H_3,
\]
 with
\begin{align*}
H_1 & =\int_{0}^{2\pi}\int_{0}^{\frac{1}{2}}\frac{dr d\theta}{|\sin(\theta+\varphi)|^{\tau}r^{q+\tau-1}|re^{i\theta}-1|^{q}},
\\
H_2 & =\int_{0}^{2\pi}\int_{\frac{1}{2}}^{\frac{3}{2}}\frac{dr d\theta}{|\sin(\theta+\varphi)|^{\tau}r^{q+\tau-1}|re^{i\theta}-1|^{q}},
\\
H_3 & =\int_{0}^{2\pi}\int_{\frac{3}{2}}^R\frac{dr d\theta}{|\sin(\theta+\varphi)|^{\tau}r^{q+\tau-1}|re^{i\theta}-1|^{q}}.
\end{align*}

For $0<r<\displaystyle\frac{1}{2}$ we have $|re^{i\theta}-1| \geq \displaystyle\frac{1}{2}$. Hence, we can find $C_1(q,\tau)>0$ such that
\[
H_{1}\leq 2^{q}\displaystyle\int_{0}^{2\pi}\int_{0}^{\frac{1}{2}}\frac{dr d\theta}{|\sin(\theta+\varphi)|^{\tau}r^{q+\tau-1}}
\leq C_1(q,\tau).
\]

For $\displaystyle\frac{1}{2}<r<\frac{3}{2}$ we have $\displaystyle\Big(\frac{2}{3}\Big)^{\tau +q} < \frac{1}{r^{\tau+q}}<2^{\tau+q}$, and so
\begin{align*}
H_{2} & \leq 2^{\tau+q}\int_{0}^{2\pi}\int_{\frac{1}{2}}^{\frac{3}{2}}\frac{dr d\theta}{|\sin(\theta+\varphi)|^{\tau}|re^{i\theta}-1|^{q}}\\
&
=2^{\tau+q}\sum_{j=1}^4\int_{\frac{(j-1)}{2}}^{\frac{j\pi}{2}}\int_{\frac{1}{2}}^{\frac{3}{2}}\frac{dr d\theta}{|\sin(\theta+\varphi)|^{\tau}|re^{i\theta}-1|^{q}}
\doteq2^{\tau+q}\sum_{j=1}^4H_{2,j}.
\end{align*}
Let $0<\theta_{0} < \pi/2$ such that
\[
1-\frac{\theta^{2}}{2}< \cos\theta < 1-\frac{\theta^{2}}{4}, \quad \textrm{for}  \quad 0< \theta < \theta_{0}.
\]
Then
\[
H_{2,1}
\!=\!\int_{0}^{\theta_0}\!\!\!\int_{\frac{1}{2}}^{\frac{3}{2}}\frac{dr d\theta}{|\sin(\theta+\varphi)|^{\tau}|re^{i\theta}-1|^{q}}+
\int_{\theta_0}^{\frac{\pi}{2}}\!\!\!\int_{\frac{1}{2}}^{\frac{3}{2}}\frac{dr d\theta}{|\sin(\theta+\varphi)|^{\tau}|re^{i\theta}-1|^{q}}
\!=\!H_{2,1}^1+H_{2,1}^2.
\]
For  $0<\theta < \theta_{0}$ and $1/2\leq r\leq3/2$ we have
\[
\frac{(r-1)^{2}+\theta^2}{4}\leq(r-1)^{2}+\frac{\theta^{2}}{4}\leq(r-1)^{2}+r\frac{\theta^{2}}{2} \leq |re^{i\theta}-1|^{2}
\]
and, consequently,
\[
\frac{1}{|re^{i\theta}-1|^{2}} \leq \frac{4}{(r-1)^{2}+\theta^2}.
\]
Hence,
\[
H_{2,1}^1 \leq
2^q\int_{0}^{\theta_{0}}\int_{\frac{1}{2}}^{\frac{3}{2}}\frac{dr d\theta}{|\sin(\theta+\varphi)|^{\tau}[(r-1)^{2}+\theta^2]^\frac{q}{2}}\doteq 2^qJ.
\]
Suppose that $0\leq\varphi\leq\pi/2$. Then, $0 <\theta +\varphi <\pi/2+\theta_{0}<\pi$ and  $|\sin(\theta+\varphi)| \geq C|\theta+\varphi|$,
for some constant $C>0$. Hence,
\begin{align*}
C^\tau J & \leq\int_{0}^{\theta_{0}}\!\!\!\int_{\frac{1}{2}}^{\frac{3}{2}}\frac{dr d\theta}{|\theta+\varphi|^{\tau}[(r-1)^{2}+{\theta^{2}}]^{\frac{q}{2}}}
\leq \int_{D((0,1);1)}\frac{dr d\theta}{|\theta+\varphi|^{\tau}[(r-1)^{2}+{\theta^{2}}]^{\frac{q}{2}}}
\\
 & =\int_0^1\!\!\int_0^{2\pi}\frac{d\rho dt}{|\varphi+\rho\sin t|^\tau \rho^{q-1}}\leq M(q,\tau),
\end{align*}
where the last inequality is obtained by Lemma \ref{lemmaintegral}.
Similar estimate can obtained in the case $\pi/2<\varphi\leq\pi$.
Therefore,
\begin{equation}\label{eq_paulo_1}
\int_{0}^{\theta_{0}}\int_{\frac{1}{2}}^{\frac{3}{2}}\frac{dr d\theta}{|\sin(\theta+\varphi)|^{\tau}|re^{i\theta}-1|^{q}}\leq\frac{2^q\,M(q,\tau)}{C^\tau}.
\end{equation}

For $\theta_{0} < \theta < \frac{\pi}{2}$ we have
\[
|re^{i\theta}-1|^{2}\geq 1-\cos\theta\geq 1-\cos\theta_0
\]
and
\begin{equation}\label{eq_paulo_2}
\begin{split}
H_{2,1}^1 & =\int_{\theta_{0}}^{\frac{\pi}{2}}\int_{\frac{1}{2}}^{\frac{3}{2}}\frac{dr d\theta}{|\sin(\theta+\varphi)|^{\tau}|re^{i\theta}-1|^{q}}
\\
& \leq \frac{1}{(1-\cos\theta_{0})^{\frac{q}{2}}}\int_{\theta_{0}}^{\frac{\pi}{2}}\frac{ d\theta}{|\sin(\theta+\varphi)|^{\tau}}
\leq \frac{C_{1}(\tau)}{(1-\cos\theta_{0})^{\frac{q}{2}}},
\end{split}
\end{equation}
for some constant $C_1(\tau)>0$.
It follows from (\ref{eq_paulo_1}) and (\ref{eq_paulo_2}), that
\begin{equation}\label{eq_paulo_3}
H_{2,1}\leq\frac{2^q\,M(q,\tau)}{C^\tau}+\frac{C_{1}(\tau)}{(1-\cos\theta_{0})^{\frac{q}{2}}}.
\end{equation}

To estimate $H_{2,2}$, we start by using a change of variable $\theta=\alpha+\pi$ in the integral to obtain
\[
H_{2,2}
=\int_{\frac{3\pi}{2}}^{2\pi}\int_{\frac{1}{2}}^{\frac{3}{2}}\frac{dr d\alpha}{|\sin(\alpha+\varphi)|^{\tau}|re^{i\alpha}+1|^{q}}.
\]
Note that $|re^{i\alpha}+1|\geq r\cos\alpha+1\geq1$, for $3\pi/2\leq\alpha\leq2\pi$; hence,
\begin{equation}\label{eq_paulo_4}
H_{2,2}\leq\int_{\frac{3\pi}{2}}^{2\pi}\frac{d\alpha}{|\sin(\alpha+\varphi)|^{\tau}}= C_2(\tau)<\infty.
\end{equation}
Similar estimates can be obtained for $H_{2,3}$ and $H_{2,4}$.
Therefore, we can find a constant $C_2(q,\tau)>0$ such that $H_2\leq C_2(q,\tau)$.

The estimation of $H_3$ is obtained as follows.
Since $r > 3/2$, then $|re^{i\theta}-1| \geq |r-1| = r-1 \geq r/3$ and
\begin{align*}
H_{3} &
\leq 3^{q}\int_{0}^{2\pi}\int_{\frac{3}{2}}^{R}\frac{dr d\theta}{|\sin(\theta+\varphi)|^{\tau}r^{2q+\tau-1}}
\\
&
\leq 3^{q}C_3(\tau)\int_{\frac{3}{2}}^R\frac{dr}{r^{2q+\tau-1}}
\leq\frac{3^{q}\,C_3(\tau)}{2q+\tau-2}\Big(\frac{2}{3}\Big)^{2q+\tau-2},
\end{align*}
where $C_3(\tau)$ is a positive constant.
This completes  the proof of (\ref{H}) in the case $\gamma=0$.

\vspace{10pt}

\noindent\textbf{Case 2:} $\gamma>0$.
In this case, $H$ (given by (\ref{H})) can be rewritten in the form $H=H_1+H_2$, where
\begin{align*}
H_1 & =\int_0^{2\pi}\int_0^{\gamma}\frac{r dr d\theta}{|\gamma+r\sin(\theta+\varphi)|^\tau r^q|re^{i\theta}-1|^q},
\\
H_2 & =\int_0^{2\pi}\int_\gamma^R\frac{r dr d\theta}{|\gamma+r\sin(\theta+\varphi)|^\tau r^q|re^{i\theta}-1|^q}.
\end{align*}
To estimate $H_1$, we consider three cases depending on the values of $\gamma$.

\noindent\textit{First case:}  $0<\gamma\le \displaystyle\frac{1}{2}$.
In this case, since $0<r<1/2$, we have $|re^{i\theta}-1|\geq1-r\geq1/2$. Hence, by Lemma \ref{lemmaintegral},
\[
H_1\leq2^q\int_0^{2\pi}\int_0^{\gamma}\frac{dr d\theta}{|\gamma+r\sin(\theta+\varphi)|^\tau r^{q-1}}
\leq 2^qM(q,\tau).
\]

\noindent\textit{Second case:}  $\displaystyle\frac{1}{2}<\gamma\le\frac{3}{2}$.
We write $H_1=H_{1,1}+H_{1,2}$, where
\begin{align*}
H_{1,1} & =\int_0^{2\pi}\int_0^{\frac{1}{2}}\frac{r dr d\theta}{|\gamma+r\sin(\theta+\varphi)|^\tau r^q|re^{i\theta}-1|^q},
\\
H_{1,2} & =\int_0^{2\pi}\int_{\frac{1}{2}}^\gamma\frac{r dr d\theta}{|\gamma+r\sin(\theta+\varphi)|^\tau r^q|re^{i\theta}-1|^q}.
\end{align*}

Note that estimation of $H_{1,1}$ is given by the previous case.

For $H_{1,2}$ we have
\begin{align*}
H_{1,2} &
\leq 2^{q-1}\int_0^{2\pi}\int_{\frac{1}{2}}^{\gamma}\frac{dr d\theta}{|\gamma-r|^\tau|re^{i\theta}-1|^q}
\\
 & =2^{q-1}\sum_{j=1}^4
\int_{\frac{(j-1)\pi}{2}}^{\frac{j\pi}{2}}\int_{\frac{1}{2}}^{\gamma}\frac{dr d\theta}{|\gamma-r|^\tau|re^{i\theta}-1|^q}
=2^{q-1}\sum_{j=1}^4 H_{1,2}^j.
\end{align*}
To estimate $H_{1,2}^1$ we write
\[
H_{1,2}^1
=\int_0^{\theta_0}\int_{\frac{1}{2}}^{\gamma}\frac{dr d\theta}{|\gamma-r|^\tau|re^{i\theta}-1|^q}+
\int_{\theta_0}^{\frac{\pi}{2}}\int_{\frac{1}{2}}^{\gamma}\frac{dr d\theta}{|\gamma-r|^\tau|re^{i\theta}-1|^q},
\]
with $0<\theta_0<\pi/2$  such that $1-\theta^2/2<\cos\theta<1-\theta^2/4$ for $0<\theta<\theta_0$.
Note that this choice of $\theta_0$ implies that $|re^{i\theta}-1|>1-\cos\theta_0$ for $\theta\geq\theta_0$. Hence,
\[
H_{1,2}^1\leq 2^q\int_0^{\theta_0}\int_{\frac{1}{2}}^{\gamma}\frac{dr d\theta}{|\gamma-r|^\tau((r-1)^2+\theta^2)^\frac{q}{2}}
+\frac{\frac{\pi}{2}-\theta_0}{(1-\cos\theta_0)^\frac{q}{2}}\frac{1}{1-\tau}.
\]
Now, by using polar coordinates $r=1+\rho \sin\phi$, $\theta =\rho\cos\phi$ in the integral, and using Lemma \ref{lemmaintegral}, we obtain
\[
\int_0^{\theta_0}\int_{\frac{1}{2}}^{\gamma}\frac{dr d\theta}{|\gamma-r|^\tau((r-1)^2+\theta^2)^\frac{q}{2}}\leq
\int_0^1\int_0^{2\pi}\frac{d\varphi d\rho}{|\gamma-1-\rho\sin\varphi|^\tau \rho^{q-1}}\leq M(q,\tau).
\]
Hence, we can find a constant $C_1(q,\tau)>0$ such that
$H_{1,2}^1\leq C_1(q,\tau)$.
Proceeding as above, we can find a constant $C_2(q,\tau)>0$ such that
$H_{1,2}^4\leq C_2(q,\tau)$.

For $H_{1,2}^2$, we have
\[
H_{1,2}^2=
\int_{\frac{3\pi}{2}}^{2\pi}\int_{\frac{1}{2}}^{\gamma}\frac{dr d\theta}{|\gamma-r|^\tau|re^{i\theta}+1|^q}
\leq\int_{\frac{3\pi}{2}}^{2\pi}\int_{\frac{1}{2}}^{\gamma}\frac{dr d\theta}{|\gamma-r|^\tau}\leq\frac{\pi}{2(1-\tau)},
\]
since $|re^{i\theta}+1|\geq1 $ for $\,\, \frac{3\pi}{2}\leq\theta\leq2\pi$.
Analogous estimates can be used to show that $H_{1,2}^3\leq \displaystyle\frac{\pi}{2(1-\tau)}$.
Therefore, we can find a constant $C_3(q,\tau)>0$ such that $H_1\leq C_3(q,\tau)$.

\vspace{10pt}

\noindent\textit{Third case:}  $\gamma>\displaystyle\frac{3}{2}$.
In this case we write $H_1=H_{1,1}+H_{1,2}$, where
\begin{align*}
H_{1,1} & =\int_0^{2\pi}\int_0^{\frac{3}{2}}\frac{r dr d\theta}{|\gamma+r\sin(\theta+\varphi)|^\tau r^q|re^{i\theta}-1|^q},
\\
H_{1,2}& =\int_0^{2\pi}\int_{\frac{3}{2}}^\gamma\frac{r dr d\theta}{|\gamma+r\sin(\theta+\varphi)|^\tau r^q|re^{i\theta}-1|^q}.
\end{align*}
Estimate for $H_{1,1}$ follows from the previous case.
For $H_{1,2}$, note that since $r\geq3/2$ we have $|re^{i\theta}-1|\geq r/2$ and
\[
H_{1,2}\leq2^q\int_0^{2\pi}\int_{\frac{3}{2}}^\gamma\frac{dr d\theta}{|\gamma+r\sin(\theta+\varphi)|^\tau r^{2q-1}}=C(q,\tau)<\infty.
\]

\vspace{10pt}

Now we estimate $H_2$.
Recall that
\[
H_2=\int_0^{2\pi}\!\!\!\int_\gamma^{R}\frac{r dr d\theta}{|\gamma+r\sin(\theta+\varphi)|^\tau r^q|re^{i\theta}-1|^q}.
\]
As done for $H_1$ we consider three cases:

\noindent\textit{First case:}  $\displaystyle\frac{3}{2}\le \gamma< R$.
Since $r>\displaystyle\frac{3}{2}$ implies $\displaystyle|re^{i\theta}-1|\geq\frac{r}{3}$, we have
\[
H_2\leq3^q\int_0^{2\pi}\!\!\!\!\int_\gamma^{R}\frac{dr d\theta}{|\gamma+r\sin(\theta+\varphi)|^\tau r^{2q-1}}
=3^q\int_0^{2\pi}\!\!\!\!\int_\gamma^{R}\frac{dr d\theta}{|\gamma+r\sin\theta|^\tau r^{2q-1}}.
\]
Hence, the estimative for $H_2$ follows from Lemma \ref{lemma_Mez-CM05}.
\vspace{10pt}

\noindent\textit{Second case:}  $\displaystyle\frac{1}{2}\le \gamma<\frac{3}{2}$.
In this case we can write $H_2=H_{2,1}+H_{2,2}$, where
\begin{align*}
H_{2,1} & =\int_{-\frac{\pi}{2}}^{\frac{3\pi}{2}}\!\!\int_\gamma^{\frac{3}{2}}\frac{r dr d\theta}{|\gamma+r\sin(\theta+\varphi)|^\tau r^q|re^{i\theta}-1|^q}=H_{2,1}^1+H_{2,1}^2,
\\
H_{2,2} & =\int_0^{2\pi}\!\!\int_{\frac{3}{2}}^{R}\frac{r dr d\theta}{|\gamma+r\sin(\theta+\varphi)|^\tau r^q|re^{i\theta}-1|^q}
\end{align*}
and,
\begin{align*}
H_{2,1}^1 & =\int_{-\frac{\pi}{2}}^\frac{\pi}{2}\int_\gamma^{\frac{3}{2}}\frac{dr d\theta}{|\gamma+r\sin\theta|^\tau r^{q-1}|re^{i(\theta-\varphi)}-1|^q},
\\
H_{2,1}^2 & =\int_{\frac{\pi}{2}}^\frac{3\pi}{2}\int_\gamma^{\frac{3}{2}}\frac{dr d\theta}{|\gamma+r\sin\theta|^\tau r^{q-1}|re^{i(\theta-\varphi)}-1|^q}.
\end{align*}
Note that estimation of $H_{2,2}$ is given by the previous case.

To estimate $H_{2,1}^1$, note that
since $1/2\le \gamma<r<3/2$ we have $1/3<\gamma/r<1$. Let $\theta_1\in(-\pi/2,-\alpha_0)$ be such that
$-\gamma/r=\sin\theta_1$, where $\sin\alpha_0=1/3$.
We have then
\[
\frac{1}{|\gamma+r\sin\theta|^\tau}=\frac{|\theta-\theta_1|^\tau}{r^\tau|\sin\theta-\sin\theta_1|^\tau}\frac{1}{|\theta-\theta_1|^\tau}
\leq\frac{A(\tau)}{r^\tau|\cos\theta_1|^\tau}\frac{1}{|\theta-\theta_1|^\tau},
\]
for some constant $A(\tau)>0$; hence, substituting $\cos\theta_1$ by $\sqrt{r^2-\gamma^2}/r$ we obtain
\[
\frac{1}{|\gamma+r\sin\theta|^\tau}\leq\frac{A(\tau)}{(r^2-\gamma^2)^\frac{\tau}{2}}\frac{1}{|\theta-\theta_1(r)|^\tau}.
\]
Therefore,
\begin{align*}
H_{2,1}^1 & \leq2^{q-1}A(\tau)\int_{-\frac{\pi}{2}}^\frac{\pi}{2}\int_\gamma^{\frac{3}{2}}\frac{dr d\theta}{(r^2-\gamma^2)^\frac{\tau}{2} |\theta-\theta_1(r)|^\tau |re^{i(\theta-\varphi)}-1|^q}
\\ &
=2^{q-1}A(\tau)\int_{-\frac{\pi}{2}-\varphi}^{\frac{\pi}{2}-\varphi}\int_\gamma^{\frac{3}{2}}\frac{dr d\theta}{(r^2-\gamma^2)^\frac{\tau}{2} |\theta-\theta_1(r)+\varphi|^\tau |re^{i\theta}-1|^q}
\\ &
\leq2^{q-1}A(\tau)\sum_{j=1}^4\int_{\frac{(j-4)\pi}{2}}^{\frac{(j-3)\pi}{2}}\int_\gamma^{\frac{3}{2}}\frac{dr d\theta}{(r^2-\gamma^2)^\frac{\tau}{2} |\theta-\theta_1(r)+\varphi|^\tau |re^{i\theta}-1|^q}
\\ &
=2^{q-1}A(\tau)\sum_{j=1}^4 H_{2,1}^{1,j}\, .
\end{align*}
We have
\[
H_{2,1}^{1,1}
=\int_{-\frac{\pi}{2}}^{0}\!\int_\gamma^{\frac{3}{2}}\frac{dr d\theta}{(r^2-\gamma^2)^\frac{\tau}{2} |\theta-\pi-\theta_1(r)+\varphi|^\tau |re^{i\theta}+1|^q}.
\]
For $-\pi/2\leq\theta\leq0$ we have $|re^{i\theta}+1|\geq1+r\cos\theta\geq1$ and
\[
H_{2,1}^{1,1}\leq\int_{-\frac{\pi}{2}}^{0}\!\int_\gamma^{\frac{3}{2}}\frac{dr d\theta}{(r^2-\gamma^2)^\frac{\tau}{2} |\theta-\pi-\theta_1(r)+\varphi|^\tau}\leq C(\tau),
\]
for some constant $C(\tau)>0$.
Similarly $H_{2,1}^{1,2}\leq\tilde{C}(\tau)$, for some $\tilde{C}(\tau)>0$.

To estimate $H_{2,1}^{1,3}$,
let  $0<\theta_0<\alpha_0/2$ such that
\[
1-\frac{\theta^2}{2}<\cos\theta<1-\frac{\theta^2}{4}, \quad \textrm{for} \quad -\theta_0\leq\theta\leq0.
\]
Note that $-\theta_0\leq\theta\leq0$ implies
\[
\theta-\theta_1(r)+\varphi\geq-\theta_0-\theta_1(r)+\varphi\geq\alpha_0-\frac{\alpha_0}{2}+\varphi\geq\frac{\alpha_0}{2}.
\]
Moreover, for $-\pi/2\leq\theta\leq-\theta_0$ we have
\[
|re^{i\theta}-1|^2\geq1-\cos\theta_0.
\]
Hence,
\begin{align*}
H_{2,1}^{1,3} & \leq\left(\frac{2}{\alpha_0}\right)^\tau\int_{-\theta_0}^0\!\int_\gamma^\frac{3}{2}\frac{drd\theta}{(r^2-\gamma^2)^\frac{\tau}{2}|re^{i\theta}-1|^q}
\\ &
+\left(\frac{1}{1-\cos\theta_0}\right)^\frac{q}{2}\int_{-\frac{\pi}{2}}^{-\theta_0}\!\int_\gamma^\frac{3}{2}\frac{dr d\theta}{(r^2-\gamma^2)^\frac{\tau}{2} |\theta-\theta_1(r)+\varphi|^\tau}
\\ &
\leq\left(\frac{2}{\alpha_0}\right)^\tau\int_{-\theta_0}^0\!\int_\gamma^\frac{3}{2}\frac{drd\theta}{(r^2-\gamma^2)^\frac{\tau}{2}|re^{i\theta}-1|^q}+C_1(\tau,q),
\end{align*}
for some constant $C_1(\tau,q)>0$.
Note that for $r>1/2$ and $-\theta_0\leq\theta\leq0$ we have
\[
|re^{i\theta}-1|^2\geq\frac{(r-1)^2+\theta^2}{4};
\]
hence, since $r\geq\gamma\geq1/2$ implies $r^2-\gamma^2\geq r-\gamma$, we have
\begin{align*}
\int_{-\theta_0}^0\!\int_\gamma^\frac{3}{2}\frac{drd\theta}{(r^2-\gamma^2)^\frac{\tau}{2}|re^{i\theta}-1|^q} & \leq
2^q\int_{-\theta_0}^0\!\int_\gamma^\frac{3}{2}\frac{drd\theta}{(r-\gamma)^\frac{\tau}{2}((r-1)^2+\theta^2)^\frac{q}{2}}
\\ &
\leq2^q\int_{D((0,1);1)}\frac{drd\theta}{(r-\gamma)^\frac{\tau}{2}((r-1)^2+\theta^2)^\frac{q}{2}}
\\ &
=2^q\int_0^1\!\!\int_0^{2\pi}\frac{drd\theta}{|1-\gamma+\rho\sin\alpha|^\frac{\tau}{2}\rho^{q-1}}
\leq C_2(\tau,q),
\end{align*}
where the last estimate follows from Lemma \ref{lemmaintegral}.
Therefore, we can find a constant $C_3(\tau,q)>0$ such that $H_{2,1}^{1,3}\leq C_3(\tau,q)$.
Similar arguments can be used to estimate $H_{2,1}^{1,4}$ and
$H_{2,1}^2$. Therefore,  we can find a constant $C(\tau,q)>0$ for which $H_{2,1}\leq C(\tau,q)$.
\vspace{10pt}

\noindent\textit{Third case:}  $\displaystyle0<\gamma\le \frac{1}{2}$.
In this case
\[
H_2\!=\!\!\int_0^{2\pi}\!\!\!\int_\gamma^{\frac{1}{2}}\!\!\frac{r dr d\theta}{|\gamma+r\sin(\theta+\varphi)|^\tau r^q|re^{i\theta}-1|^q}
+\int_0^{2\pi}\!\!\!\int_{\frac{1}{2}}^{R}\!\!\frac{r dr d\theta}{|\gamma+r\sin(\theta+\varphi)|^\tau r^q|re^{i\theta}-1|^q}.
\]
The second integral can be estimate as before. Hence, it is enough to estimate the first integral.
Since $r<1/2$ implies $|re^{i\theta}-1|\geq|r-1|\geq1/2$, we have
\[
\int_0^{2\pi}\int_\gamma^{\frac{1}{2}}\frac{r dr d\theta}{|\gamma+r\sin(\theta+\varphi)|^\tau r^q|re^{i\theta}-1|^q}
\leq2^q\int_0^{2\pi}\int_0^{\frac{1}{2}}\frac{dr d\theta}{|\gamma+r\sin(\theta+\varphi)|^\tau r^{q-1}},
\]
which can be estimated applying Lemma \ref{lemmaintegral}.

This completes the proof of the case $\gamma >0$ and of the Lemma.
\end{proof}

%%%%%%%%%%%%%%%%%%% end of some lemmas%%%%%%%%%%%%%%%%%%%%%%%

\section{An Integral Operator}

Let $L$ be a vector field defined in an open set $\tilde{\Omega} \subset \mathbb{R}^{2}$ and $\Omega$ be a
relatively compact subset of $\tilde{\Omega}$ ($\overline{\Omega} \subset \tilde{\Omega}$)
with $\partial\Omega$ piecewise of class $C^{1}$. Suppose that $L$ satisfies conditions (i), (ii), and (iii). Let
\begin{equation}
Z: \overline{\Omega} \to Z(\overline{\Omega}) \subset \mathbb{C}
\end{equation}
be a global first integral of $L$ of class $C^{1+\epsilon}$ as in Proposition 4.
As noted in section \ref{preliminaries}, $L$ is a multiple of the Hamiltonian of $Z$. In fact, we will assume that
\begin{equation}
L = Z_{x}\frac{\partial}{\partial y} - Z_{y}\frac{\partial}{\partial x}.
\end{equation}
For $f \in L^{1}(\Omega)$, define the integral operator
\begin{equation}\label{integral_operator}
T_{Z}f(x,y)={\frac{1}{2\pi i}}\int_{\Omega}\frac{f(\xi, \eta)}{Z(\xi, \eta)-Z(x,y)}d\xi d\eta.
\end{equation}
Since $L$ satisfies condition (iii), then for every $\Sigma_{j}$, connected component of $\Sigma$, there exists $\sigma_{j}>0$ such that (\ref{Z_normalized}) holds. Let
\begin{equation}\label{sigma}
\sigma= \max \sigma_{j}.
\end{equation}
We have the following theorem:

\begin{theorem}\label{continoperatz}
Let $\sigma$ be given by {\em(\ref{sigma})} and let $f \in L^{p}(\Omega)$, with $p>2+\sigma$. Then, there exists a constant $M(p,\sigma, \Omega)>0$ such that the operator $T_{Z}$ satisfies
\begin{equation*}
|T_{Z}f(x,y)| \leq M(p,\sigma,\Omega)||f||_{p}, \quad \forall \, (x,y) \in \Omega.
\end{equation*}
\end{theorem}
\begin{proof}
Let $(x,y) \in \Omega$. Then
\begin{equation*}
|T_{Z}f(x,y)| \leq \frac{1}{2\pi}\int_{\Omega}\frac{|f(\xi, \eta)|}{|Z(\xi, \eta)-Z(x,y)|}d\xi d\eta \leq \frac{||f||_{p}}{2\pi}\Big(\int_{\Omega}\frac{d\xi d\eta}{|Z(\xi, \eta)-Z(x,y)|^{q}}\Big)^{\frac{1}{q}}
\end{equation*}
where $q= p/(p-1)$. Since $Z$ is a $C^{1+\epsilon}$-diffeomorphism outside of the characteristic set $\Sigma$, then it follows from the normalization
given in Proposition \ref{normalization} and use of partition of unity that in order to have
\[
\int_{\Omega}\frac{d\xi d\eta}{|Z(\xi, \eta)-Z(x,y)|^{q}} \leq M(p, \sigma, \Omega)
\]
it is enough to have the inequality when $\Omega$ is replaced by $\Omega \cap \mathcal{U}$, when $\mathcal{U}$ is a open set where the normalization holds.
For the purpose of estimating the integral, there is no loss of generality in assuming that the vector field can be transformed into the normal form
$L_\sigma$ in the open set $\mathcal{U}$ rather than in each connected component $\mathcal{U}^+$ and $\mathcal{U}^-$ of $\mathcal{U}\backslash \Sigma$ (since the diffeomorphisms extend across the boundary by Proposition \ref{normalization}).

Let $w \in \Sigma_{j} \subset \Sigma$. There exists a $C^{1}$-diffeomorphism
\[
\Phi:D(0,\delta) \to U_{w}=\Phi(D(0,\delta)) \subset \mathbb{R}^{2}
\]
with $\Phi(0)=w$, and a holomorphic function $H$ defined on $\overline{Z(U_{w})}$, with $H'(\zeta)\neq 0$ for $\zeta \in \overline{Z(U_{w})}$, such that
\[
Z\circ \Phi(s,t) = H(Z_{j}(s,t)),
\]
where $Z_{j}(s,t)=s+i\displaystyle\frac{t|t|^{\sigma_{j}}}{1+\sigma_{j}}$.
\\
Hence for $(x,y) \in U_{w}$
\[
\int_{\Omega\cap U_{w}}\frac{d\xi d\eta}{|Z(\xi, \eta)-Z(x,y)|^{q}} \leq \int_{D(0,\delta)}\frac{|D\Phi(s,t)|ds dt}{|H(Z_{j}(s, t))-H(Z_{j}(s^{\circ},t^{\circ}))|^{q}}
\]
\[
\leq \frac{\max|D\Phi(s,t)|}{\min|H'(\zeta)|^{q}}\int_{D(0,\delta)}\frac{ds dt}{|Z_{j}(s, t)-Z_{j}(s^{\circ},t^{\circ})|^{q}}.
\]
By using the change of variables $\xi = s$ and $\eta=\displaystyle\frac{t|t|^{\sigma_{j}}}{1+\sigma_{j}}$ we obtain
\begin{align*}
(1+\sigma_j)^{\tau_j}\int_{D(0,\delta)} \frac{ ds dt }{|Z_j(s,t)-Z_j(s^\circ,t^\circ)|^{q}}
 & =\int_{Z_j(D(0,\delta))}\frac{d\xi d\eta}{|\eta|^{\tau}|\zeta-z|^{q}}\\
 & \leq\int_{D(z,\textsf{d})}\frac{d\xi d\eta}{|\eta|^{\tau}|\zeta-z|^{q}},
\end{align*}
where $\zeta = \xi +i\eta$, $z=Z_{j}(s^{\circ}, t^{\circ})$, and $\textsf{d}=\textrm{diam}(Z_j(D(0,\delta)))$.

Now, the use of polar coordinates $(r,\theta)$ given by
$\xi =r{\rm cos}\theta + \Re(z)$ and $\eta = r\sin\theta + \Im(z)$ and of Lemma \ref{lemma_holder} give
\[
\int_{D(z,\textsf{d})}\frac{d\xi d\eta}{|\eta|^{\tau}|\zeta-z|^{q}}=\int_{0}^{2\pi}\int_{0}^{\textsf{d}}\frac{drd\theta}{|r\sin\theta+\Im(z)|^{\tau}r^{q-1}}
\leq M(q,\tau)\,\textsf{d}^{2-\tau-q}.
\]
\end{proof}

%%%%%%%%%%%%%%% end of theorem %%%%%%%%%%%%%%%%

\begin{proposition}\label{1q2menoslambdaquandofestaeml1}
Let $\sigma$ be given by {\em(\ref{sigma})} and let  $f \in L^{1}(\Omega)$. Then, $Tf \in L^{q}(\Omega)$, for any $1\leq q <  2 -\sigma/(\sigma+1)$.
\end{proposition}
\begin{proof}
Let $f \in L^{1}(\Omega)$ and let $g \in L^{p}(\Omega)$, with $p>2+\sigma$. It follows from Lemma \ref{lemmaintegral} that the function
\[
g_{1}(x,y)=\displaystyle\int_{\Omega}|g(\xi,\eta)|\frac{d\xi d\eta}{|Z(\xi,\eta)-Z(x,y)|}
\]
is bounded. Hence, $fg_{1}\in L^{1}(\Omega)$.
By applying Fubini's Theorem we obtain
\[
\int_{\Omega}|f(x,y)|g_{1}(x,y)dxdy=\int_{\Omega}|f(x,y)|\left(\int_{\Omega}|g(\xi,\eta)|\frac{d\xi d \eta}{|Z(\xi,\eta)-Z(x,y)|}\right)dxdy
\]
\[
\quad\quad\quad\quad\quad\quad\quad\quad\quad\quad\,
=\int_{\Omega}|g(\xi,\eta)|\left(\int_{\Omega}|f(x,y)|\frac{dxdy}{|Z(\xi,\eta)-Z(x,y)|}\right)d\xi d\eta
\]
\[
=\int_{\Omega}|g(\xi,\eta)|f_{1}(\xi , \eta)d\xi d\eta,
\,\,\,
\]
where
\[
f_{1}(\xi,\eta)=\displaystyle\int_{\Omega}|f(x,y)|\frac{dxdy}{|Z(x,y)-Z(\xi,\eta)|}.
\]
Hence $|g|f_{1} \in L^{1}(\Omega)$.
Since $g\in L^{p}(\Omega)$ is arbitrary, it follows from the converse of the H\"{o}lder inequality (see, for instance, \cite{Yeh}; also, \cite{Leach}) that $f_{1} \in L^{q}(\Omega)$,
for $q=p/(p-1)$. Note that $p>2+\sigma$ implies
$1<q <2-\sigma/(\sigma+1)$. Also, note that $|Tf|\leq f_1$. Therefore, $Tf\in L^{q}(\Omega)$, for any $1<q <2-\sigma/(\sigma+1)$.
\end{proof}

The following Lemma is a direct consequence of Green's Theorem.

\begin{lemma}
Let $w\in C(\overline{\Omega})\cap C^{1}({\Omega})$. Then
\begin{equation}\label{lgauss}
\int_{\Omega}Lw \,dxdy=-\int_{\partial \Omega}w\,dZ(x,y).
\end{equation}
\end{lemma}

\begin{proposition}\label{reprefhhfhfhfhfh2}
Let $w \in C(\overline{\Omega})\cap C^{1}(\Omega)$. Then, for all $(x,y) \in\Omega$, we have
\begin{equation*}%\label{fefff2}
w(x,y) =\frac{1}{2 \pi i}\int_{\partial \Omega}\frac{w(\alpha,\beta)}{Z(\alpha,\beta)-Z(x,y)}dZ(\alpha,\beta)
+\frac{1}{2 \pi i}\int_{\Omega}\frac{Lw(\alpha,\beta)}{Z(\alpha,\beta)-Z(x,y)} \ d\alpha d\beta.
\end{equation*}
\end{proposition}
\begin{proof}
Let $(x_{0},y_{0}) \in \Omega$ fixed.
Set $z_{0}= Z(x_{0},y_{0})$ and let $\epsilon >0$ such that $\overline{D_{\epsilon}} \subset Z(\Omega)$, where $D_{\epsilon}=D(z_{0}, \epsilon)$.
Define $K_{\epsilon}=Z^{-1}(D_{\epsilon})$ and $\Omega_{\epsilon}=\Omega \setminus K_{\epsilon}$.
Let
\[
f(x,y)=\frac{w(x,y)}{Z(x,y)-z_{0}}.
\]
We have $f \in C(\overline{\Omega_{\epsilon}}) \cap C^{1}(\Omega_{\epsilon})$. Hence, by \eqref{lgauss},
\begin{equation}\label{eqysysys1}\begin{split}
\int_{\Omega_{\epsilon}}\frac{Lw(x,y)}{Z(x,y)-z_{0}} \ dxdy
& =-\int_{\partial \Omega_{\epsilon}}\frac{w(x,y)}{Z(x,y)-z_{0}}dZ(x,y)\\
&
=-\int_{\partial \Omega}\frac{w(x,y)}{Z(x,y)-z_{0}}dZ(x,y) + \int_{\partial K_{\epsilon}}\frac{w(x,y)}{Z(x,y)-z_{0}}dZ(x,y);
\end{split}\end{equation}
Since
\[
\int_{\partial K_{\epsilon}}f(x,y)dZ(x,y) =\int_{Z(\partial K_{\epsilon})}(f\circ Z^{-1})(\zeta)d\zeta = \displaystyle\int_{\partial B_{\epsilon}}\frac{\tilde{w}(\zeta)}{\zeta -z_{0}}d\zeta,
\]
where $\tilde{w}= w\circ Z^{-1}$,
then by using polar coordinates $\zeta = z_{0} + \epsilon e^{i\theta}$, $\theta \in [0,2\pi]$, we obtain
\[
\int_{\partial D_{\epsilon}(z_{0})}\frac{\tilde{w}(\zeta)}{\zeta - z_{0}}d \zeta =\int_{0}^{2\pi} \frac{\tilde{w}(z_{0}
+\epsilon e^{i\theta})}{\epsilon e^{i\theta}}i\epsilon e^{i\theta}d\theta \to 2\pi i \tilde{w}(z_{0}), \ \  \mbox{as} \ \ \epsilon \to 0.
\]
Therefore,
\begin{equation}\label{eqysysys2}
\int_{\partial K_{\epsilon}}\frac{w(x,y)}{Z(x,y)-z_{0}}dZ(x,y) \to 2\pi i w(x_{0},y_{0}) ,\ \  \mbox{as} \ \ \epsilon \to 0.
\end{equation}

On the other hand, as done in the proof of Theorem \ref{continoperatz},
\[
(x,y) \mapsto\frac{1}{Z(x,y) - z_{0}} \in L^{q}(\Omega)\subset L^{1}(\Omega).
\]
Hence,
\begin{equation}\label{eqysysys3}
\int_{\Omega_{\epsilon}}\frac{Lw(x,y)}{Z(x,y)-z_{0}} \ dxdy \to \int_{\Omega}\frac{Lw(x,y)}{Z(x,y)-z_{0}} \ dxdy, \quad as \quad \epsilon \to 0.
\end{equation}
The Proposition follows
from \eqref{eqysysys1}, \eqref{eqysysys2} and \eqref{eqysysys3}.
\end{proof}

%%%%%%%%%%%%%%%%% end of proposition %%%%%%%%%%%%

\begin{theorem}\label{tsolucao}
If $f \in L^1(\Omega)$ then $T_Zf$ satisfies $L(T_Zf)=f$ in $\Omega$.
\end{theorem}
\begin{proof}
Let $f \in L^1(\Omega)$. By Proposition \ref{1q2menoslambdaquandofestaeml1}, $T_Zf\in L^q(\Omega)$, $1<q<2-\sigma/(\sigma+1)$. Hence, by applying Proposition
\ref{reprefhhfhfhfhfh2} we have, for $\phi\in C^\infty_0(\Omega)$
\begin{align*}
\langle L(T_Zf), \phi \rangle & =-\int_{\Omega}T_Zf(x,y)L\phi(x,y)dxdy\\ &
=-\frac{1}{2 \pi i} \int_{\Omega} \left(\int_{\Omega}\frac{f(\xi,\eta)}{Z(\xi,\eta)-Z(x,y)} \ d\xi d\eta\right)L\phi(x,y)dxdy
\\ &
=\int_{\Omega} f(\xi,\eta)\left(-\frac{1}{2 \pi i} \int_{\Omega}\frac{L\phi(x,y)}{Z(x,y)-Z(\xi,\eta)} \ dxdy \right)d\xi d\eta
\\ &
=\int_{\Omega} f(\xi,\eta)\phi(\xi,\eta)d\xi d\eta=\langle f,\phi\rangle\, .
\end{align*}
Therefore, $L(T_Zf)=f$ in $\Omega$.
\end{proof}

\begin{example} For $p\le 2+\sigma$, there exist $f\in L^p(\Omega )$ such that equation $Lu=f$ has no bounded solutions.
 Indeed, for the standard vector field $L=\partial_t-i|t|^\sigma \partial_x$,
 the function $v:\overline{\Omega}\rightarrow\mathbb{C}$ defined by $v(x,t)=\ln|\ln|Z(x,t)||$ is not bounded but solves
 $Lv=f$ with
\[
f(x,t)=\frac{-i|t|^\sigma}{\overline{Z(x,t)}\,\ln|Z(x,t)|}\in L^p(\Omega), \quad \textrm{for\ any}\quad  1\le p\le 2+\sigma.
\]
That $f\in L^p$ with $p\le 2+\sigma$ follows from the fact that $\displaystyle \int_0^a \frac{dr}{r^s|\ln r|^p}\, <\, \infty$ if and only if $s\le 1$.
For general vector fields, similar examples can be produced thanks to the normalization near points on $\Sigma$.
\end{example}

\begin{remark}
Cauchy type integral operators were used in {\cite{Mez-Mem}} and {\cite{Mez-CV08}} in connection
with other types of vector fields.
\end{remark}

%%%%%%%%%%%%%%% holder section %%%%%%%%%%%%%%%%%%

\section{H\"{o}lder continuity of solutions}

In this section we prove that the solutions of $Lu=f$ are H\"older continuous if $f\in L^p$, with $p>2+\sigma$.
Let $\Sigma_1,\cdots,\Sigma_N$ be the connected components of $\Sigma$ and $\sigma_1,\cdots,\sigma_N$ be respectively the types of $L$ along
$\Sigma_1,\cdots,\Sigma_N$. Recall that
$\sigma=\displaystyle\max_{1\leq j\leq N}\{\sigma_j\}$, $ \tau=\displaystyle\frac{\sigma}{\sigma+1}$, and
for $p>2+\sigma$ and $q=\displaystyle\frac{p}{p-1}$ we have $q<2-\tau$.

\begin{theorem}\label{holder_continuity}
Let $f\in L^p(\Omega)$, with $p>2+\sigma$. Let $q$ and $\tau$ be as given above.
If $u$ satisfies $Lu=f$  in $\Omega$, then $u\in C^\alpha (\Omega )$ with
$\alpha = \displaystyle\frac{2-q-\tau}{q}$.
\end{theorem}

\begin{proof} Since $L$ is hypocomplex, it follows from Theorem \ref{tsolucao} that if $u$ solves $Lu=f$ in $\Omega$,
then $u=T_Zf +H(Z)$, where $T_Z$ is the integral operator defined in section 4 and $H$ is a holomorphic function
defined on $Z(\Omega)$. Thus to prove the Theorem, we need only prove that
$T_Zf\in C^\alpha (\ov{\Omega})$.

Let $f\in L^p(\Omega)$, with $p>2+\sigma$, and let $(x_0,y_0), (x_1,y_1)\in\overline{\Omega}$.
Set $z_1=Z(x_1,y_1)$ and $z_0=Z(x_0,y_0)$. Then
\begin{align*}
|T_Zf(x_1,y_1)-T_Zf(x_0,y_0)| &
\leq \frac{|z_1-z_0|}{2\pi}\int_\Omega \frac{|f(\alpha,\beta)|d\alpha d\beta}{|Z(\alpha,\beta)-z_1||Z(\alpha,\beta)-z_0|}\\
& \leq \frac{|z_1-z_0|}{2\pi} ||f||_p J^{\frac{1}{q}},
\end{align*}
where
\begin{equation}\label{J}
J=\int_\Omega \frac{d\alpha d\beta}{|Z(\alpha,\beta)-z_1|^q|Z(\alpha,\beta)-z_0|^q}.
\end{equation}
To prove the Theorem, it is enough to show that
\begin{equation}\label{final_inequality}
J\leq C_1|z_1-z_0|^{2-2q-\tau}+C_2,
\end{equation}
for some constants $C_1,C_2>0$.

For each $j=1,\cdots,N$, let $V_j$ be a tubular neighborhood of $\Sigma_j$ such that $V_j\cap V_k =\emptyset$ for $j\neq k$
and such that
\[
V_j=\bigcup_{k=1}^{M_j}V_{jk}, \quad j=1,\cdots,N,
\]
where each $V_{jk}$ is an open subset where $L$ can be transformed into the standard vector field $L_{\sigma_j}$, with first integral
\[
Z_j(s,t)=s+i\frac{t|t|^{\sigma_j}}{1+\sigma_j}
\]
on each side $V_{jk}^\pm$ of the characteristic curve $\Sigma\cap V_{jk}$ (see Proposition \ref{normalization}).
Hence, we can assume that there exists a diffeomorphisms of class $C^1$
\[
\Phi^\pm_{jk}:D^\pm(0,R)\rightarrow V^\pm_{jk}
\]
such that $Z\circ\Phi^\pm_{jk}=H^\pm_{jk}\circ Z_j$, where $H^\pm_{jk}$ is a holomorphic function defined in $Z_{j}(\overline{D^\pm(0,R)})$
and has a $C^1$ extension up to the boundary. Moreover, $H'_{jk}(\zeta)\geq C_{jk}>0$ on $\ov{D^\pm(0,R)}$.

Since $L$ is elliptic outside $\Sigma$, then we can use partition of unity to reduce the problem of proving (\ref{final_inequality})
into that of proving the inequality  when
$\Omega$ is replaced by $V^\pm_{jk}$, with $(x_0,y_0),(x_1,y_1)\in V^\pm_{jk}$. In fact for the estimation of the integral, there is no
loss of generality in assuming that $L$ can be transformed into the standard vector in $V_{jk}$ and not only in each $V_{jk}^\pm$ separately.

Set $\Phi_{jk}(s_\ell,t_\ell)=(x_\ell,y_\ell)$, $\ell=0,1$. We have
\[
J_{jk}=\int_{V_{jk}} \frac{d\alpha d\beta}{|Z(\alpha,\beta)-Z(x_1,y_1)|^q|Z(\alpha,\beta)-Z(x_0,y_0)|^q}
\]
\[
=\int_{D(0,R)}\frac{|\det D\Phi_{jk}(s,t)| dsdt}{|H_{jk}(Z_j(s,t))-H_{jk}(Z_j(s_1,t_1))|^q|H_{jk}(Z_j(s,t))-H_{jk}(Z_j(s_0,t_0))|^q}
\]
Hence,
\[
J_{jk}\leq M_{jk}\int_{D(0,R)}\frac{ dsdt}{|(Z_j(s,t)-Z_j(s_1,t_1)|^q|Z_j(s,t)-Z_j(s_0,t_0)|^q}\doteq M_{jk}Q_{jk},
\]
where
\[
M_{jk}=\frac{\max_{D(0,R)}\{|\det D\Phi_{jk}(s,t)|\}}{\min_{Z(D(0,R))}|H'_{jk}(\zeta)|^q}.
\]
By using the change of variables $\xi=s$ and $\eta=\displaystyle\frac{t|t|^{\sigma_j}}{\sigma_j+1}$ we have
\[
Q_{jk}=\frac{1}{(\sigma_j+1)^{\tau_j}}\int_{Z_j(D(0,R))}\frac{d\xi d\eta}{|\eta|^{\tau_j}|\zeta-w_1|^q|\zeta-w_0|^q},
\]
where $w_\ell=Z_j(s_\ell,t_\ell)$, $\ell=0,1$.
We can assume, without loss of generality, that $\eta_0>0$ and $w_1-w_0=|w_1-w_0|e^{i\varphi}$, with $0\leq\varphi\leq\pi$.
We have
\begin{align*}
Q_{jk} & \leq\frac{1}{(\sigma_j+1)^{\tau_j}}\int_{D(w_0;\textsf{d})}\frac{d\xi d\eta}{|\eta|^{\tau_j}|\zeta-w_1|^q|\zeta-w_0|^q}
\\ &
=\frac{1}{(\sigma_j+1)^{\tau_j}}\int_{D(0;\textsf{d})}\frac{d\xi d\eta}{|\eta_0+\eta|^{\tau_j}|\zeta-(w_1-w_0)|^q|\zeta|^q},
\end{align*}
where $\textsf{d}=\textrm{diam}(Z_j(D(0,R)))$.
After using the change of variables $\zeta=|w_1-w_0|\mu$, with $\mu=a+ib$ and $w_1-w_0=|w_1-w_0|e^{i\varphi}$, we obtain
\[
Q_{jk}\leq\frac{1}{(\sigma_j+1)^{\tau_j}}\int_{D\left(0;\frac{\textsf{d}}{|w_1-w_0|}\right)}\frac{|w_1-w_0|^{2-\tau-2q}}{|\frac{\eta_0}{|w_1-w_0|}+b|^\tau|\mu-e^{i\varphi}|^q|\mu|^q}da db.
\]
Finally, the polar coordinates $\mu=\zeta e^{i\varphi}$, $\zeta=re^{i\theta}$ and $\gamma=\eta_0/|w_1-w_0|$ and the use of Lemma \ref{lemma_holder} allow us to obtain the estimate
\begin{align*}
Q_{jk} & \leq\frac{|w_1-w_0|^{2-{\tau_j}-2q}}{(\sigma_j+1)^{\tau_j}}\int_0^{2\pi}
\int_0^{\frac{\textsf{d}}{|w_1-w_0|}}\frac{r dr d\theta}{|\gamma+r\sin(\theta+\varphi)|^{\tau_j} r^q|re^{i\theta}-1|^q}d\xi d\eta
\\ \\ &
\leq\frac{C(q,\tau_j)}{(\sigma_j+1)^{\tau_j}}|w_1-w_0|^{2-\tau_j-2q},
\end{align*}

To finish the proof, note that
\[
|w_1-w_0|=|H_{jk}^{-1}(Z(x_1,y_1))-H_{jk}^{-1}(Z(x_0,y_0))|
\geq A_{jk}|Z(x_1,y_1)-Z(x_0,y_0)|,
\]
for some constant $A_{jk}>0$, since $H_{jk}$ is a biholomorphism. Hence,
\[
Q_{jk} \leq \tilde{A}_{jk}|Z(x_1,y_1)-Z(x_0,y_0)|^{2-\tau_j-2q} \leq \tilde{A}_{jk}|Z(x_1,y_1)-Z(x_0,y_0)|^{2-\tau-2q}.
\]
Therefore, we can find a constant $M(q,\tau,\Omega)>0$ such that
\[
|T_Zf(x_1,t_1)-T_Zf(x_0,t_0)|\leq\|f\|_pM(q,\tau,\Omega)|Z(x_1,t_1)-Z(x_0,t_0)|^\frac{2-\tau-q}{q}
\]
This completes the proof.
\end{proof}

\section{A semilinear equation and a similarity principle}

In this section we consider the semilinear equation
\begin{equation}\label{dzff}
Lu= F(x,y,u),
\end{equation}
where $L$ is the hypocomplex vector field as in the previous sections defined in an open set
$\tilde{\Omega}\subset \mathbb{R}^2$.

Let $\Omega$ be a relatively compact open subset of $\tilde{\Omega}$ and let $\Psi\,\in\,L^{p}(\overline{\Omega};\overline{\mathbb{R}}_+)$, $p>2+\sigma$.
 We define the space $\mathcal{F}_\Psi^\alpha$ to be the set of functions
$F: \overline{\Omega}\times \mathbb{C} \to \mathbb{C}$ satisfying
\begin{itemize}
\item $F(.,\zeta) \in L^{p}(\overline{\Omega})$, for every $\zeta\in\mathbb{C}$;

\item $|F(x,t,\zeta_{1})-F(x,t,\zeta_{2})| \leq \Psi(x,t)|\zeta_{1}-\zeta_{2}|^{\alpha}$, $0<\alpha\leq 1$,
for all $\zeta_{1},\zeta_{2}\in \mathbb{C}$.
\end{itemize}
Throughout this section, we will assume that $q$ is the H\"{o}lder conjugate of $p$, that $\sigma $ is given by (\ref{sigma}),
$\tau =\displaystyle\frac{\sigma}{\sigma +1}$, and $\beta=\displaystyle \frac{2-q-\tau}{q}$.

\begin{theorem}\label{theorem_1_Lu=F}
Let $F\in\mathcal{F}_h^\alpha$. Then:
\begin{enumerate}
\item If $0< \alpha <1$, equation \eqref{dzff} has a solution $u\in C^{\beta}(\overline{\Omega})$.
\item If $\alpha =1$,  for every $(x,t) \in \Omega$, there exist an open subset $\mathcal{U}\subset\Omega$, with $(x,t) \in\mathcal{U}$,
such that equation \eqref{dzff} has a solution
$u \in C^\beta(\mathcal{U})$.
 If, moreover, the constant $M(p,\sigma,\Omega)$ appearing in Theorem \ref{continoperatz}
 satisfies  $M(p,\sigma,\Omega)||\Psi||_{p}<1$, then \eqref{dzff}
 has a solution in $C^\beta(\ov{\Omega})$.
\end{enumerate}
\end{theorem}
\begin{proof}
Let $C(\overline{\Omega})$ be the Banach space of continuous functions in $\overline{\Omega}$ with  norm
\[
||u||_{\infty}=\sup\{|u(x,t)| \ ; \ (x,t) \in \overline{\Omega}\}.
\]
For a fixed $M>0$, let $C_{M}(\overline{\Omega})$ be the closed subset of $C(\overline{\Omega})$ given by
\[
C_{M}(\overline{\Omega})= \{u\in C(\overline{\Omega}) \ ; \ ||u||_{\infty} \leq M\}.
\]

Suppose that $0<\alpha<1$. Then, for $M$ sufficiently large, we have that
\begin{equation}\label{vintebhhh}
M(p,\sigma,\Omega)\{||\Psi||_{p}M^{\alpha}+||F(\cdot,0)||_{p}\} \leq M,
\end{equation}
where $M(p,\sigma,\Omega)$ is given by Theorem \ref{continoperatz}.

Consider the operator $P: C_{M}(\overline{\Omega}) \to C_{M}(\overline{\Omega})$ defined by
\[
Pu(x,y)=T_Z(F(x,y,u(x,y)))\, ,
\]
where $T_Z$ is given by (\ref{integral_operator}). The operator $P$ is well-defined.
Indeed, for $F\in\mathcal{F}_\Psi^\alpha$ and $u\in C_{M}(\overline{\Omega})$ we have
\begin{equation}\label{equationnnnnntresdsds}
||F(\cdot,u)||_{p} \leq ||\Psi||_{p}M^{\alpha} + ||F(\cdot,0)||_{p}\, ,
\end{equation}
and it follows from \eqref{vintebhhh}, \eqref{equationnnnnntresdsds}, and Theorem \ref{continoperatz} that
for $u \in C_{M}(\overline{\Omega})$, we have
\begin{equation}\label{vintebhh}
|Pu(x,y)|\leq M(p,\sigma,\Omega)\{||\Psi||_{p}M^{\alpha}+||F(\cdot,0)||_{p}\} \leq M\, .
\end{equation}
Moreover, for $(x_{1},y_{1}), (x_{2},y_{2}) \in \overline{\Omega}$, we have
\begin{align*}
|Pu(x_{1},y_{1})-Pu(x_{2},y_{2})| & \leq \left|T_Z(F(\cdot,u))(x_{1},y_{1})-T_Z(F(\cdot,u))(x_{2},y_{2})\right|
\\  &
\leq C(p,\sigma,\Omega)||F(\cdot,u)||_{p}|(x_{1},y_{1})-(x_{2},y_{2})|^\beta\\
&
\leq C(p,\sigma,\Omega)\{||\Psi||_{p}M^{\alpha} + ||F(\cdot,0)||_{p}\}|(x_{1},y_{1})-(x_{2},y_{2})|^\beta\, ,
\end{align*}
where $C(p,\sigma,\Omega)$ is given by Theorem \ref{holder_continuity}.

For $C=C(p,\sigma,\Omega)\{||\Psi||_{p}M^{\alpha} + ||F(\cdot,0)||_{p}\}$, define $\Lambda_{M,C}$
as the set of all functions $v\in C_{M}(\overline{\Omega})$ satisfying
\[
|v(x_{1},y_{1})-v(x_{2},y_{2})| \leq C|(x_{1},y_{1})-(x_{2},y_{2})|^\beta , \ \ \forall (x_{1},y_{1}), (x_{2},y_{2}) \in \overline{\Omega}.
\]
$\Lambda_{M,C}$ is a nonempty convex subset of $C_{M}(\overline{\Omega})$. Also, as a consequence of
Ascoli-Arzel\'{a}'s Theorem, $\Lambda_{M,C}$ is compact. Moreover, $Pu \in \Lambda_{M,C}$, for all $u \in C_{M}(\overline{\Omega})$.

The operator $P$ is continuous. Indeed, for $u,v \in C_{M}(\overline{\Omega})$, we have
\begin{align*}
|Pu(x,y)-Pv(x,y)| & \leq \left|T_Z(F(\cdot,u)-F(\cdot,v))(x,y)\right| \\
 & \leq M(p,\sigma,\Omega)||F(\cdot,u)-F(\cdot,v)||_{p}
\\ &
\leq M(p,\sigma,\Omega)||\Psi||_{p}||u-v||_{\infty}^{\alpha};
\end{align*}
hence, the restriction $P: \Lambda_{M,C} \to \Lambda_{M,C}$ is continuous.
Therefore, Shauder fixed point Theorem, implies that there exists $u \in \Lambda_{M,C}$ such that $Pu =u$. The fixed point
\[
u= T(F(\cdot,u))\in C^{\beta}(\overline{\Omega})
\]
satisfies $Lu(x,y)=F(x,y,u(x,y))$.\vspace{5pt}

Next, suppose that $\alpha=1$. In this case, for every $(x,y)\in\Omega$ we can find an open
$\mathcal{U}\subset\Omega$, with $(x,y)\in\mathcal{U}$, such that
\[
M(p,\sigma,\mathcal{U})||\Psi||_{p} < 1.
\]
If $M>0$ is taken sufficiently large, \eqref{vintebhhh} holds when $\Omega$ is replaced by $\mathcal{U}$ .
The same argument as the one used when $\alpha <1$, shows the existence of $u\in C^{\beta}({\mathcal{U}})$ satisfying \eqref{dzff} in $\mathcal U$.

\end{proof}

As a direct consequence of the Theorem \ref{theorem_1_Lu=F}, we have the following Corollary.

\begin{corollary}
Let $a,b, f\in L^p_{loc}(\mathbb{R}^2)$, with $p>2+\sigma$.
Every $(x,y)\in\mathbb{R}^2$ has an open neighborhood $\mathcal{U}\subset\mathbb{R}^2$ such that
equation
\[ Lu=au+b\overline{u}+f\]
has a solution  $u\in C^{\beta}(\overline{\mathcal{U}})$.
\end{corollary}

Now, consider $F$ given by
\begin{equation}\label{dzfff}
F(x,y,u) =g(x,y)H(x,y,u)+f(x,y)
\end{equation}
where $f,g \in L^{p}({\Omega})$, $p>2+\sigma$, and $H:\overline{\Omega}\times\mathbb{C}\to \mathbb{C}$ is continuous and bounded,
with $||H||_\infty <K$ for some positive constant $K$.

\begin{theorem}\label{abfcont}
Let $F$ be given by \eqref{dzfff}. Then, equation $Lu=F(x,y,u)$ has a solution $u\in C^{\beta}(\overline{\Omega})$.
\end{theorem}
\begin{proof}
Consider the operator $P: C(\overline{\Omega}) \to C(\overline{\Omega})$ defined by
\[
Pu(x,y)=T_Z(gH(\cdot,u)+f)(x,y).
\]
Since $H(\cdot,u) \in L^{\infty}({\Omega})$ we have $gH(\cdot,u)+f \in L^{p}({\Omega})$. It follows
from Theorem \ref{continoperatz}, that for every $(x,y)\in\ov{\Omega}$ we have
\begin{align*}
|Pu(x,y)|& \leq M(p,\sigma,\Omega)||gH(\cdot,u)+f||_{p}\\
& \leq M(p,\sigma,\Omega)\{||g||_{p}||H(\cdot,u)||_{\infty}+||f||_{p}\}
\\ &
\leq M(p,\sigma,\Omega)\{||g||_{p}K+||f||_{p}\}\doteq M\, .
\end{align*}
 Moreover, it follows from  Theorem \ref{holder_continuity}, that
for $(x_{1},y_{1}), (x_{2},y_{2}) \in \overline{\Omega}$ and $u \in C(\overline{\Omega})$ there is $C(p,\sigma,\Omega)>0$ such that
\begin{align*}
|Pu(x_{1},y_{1})-Pu(x_{2},y_{2})| & \leq
\left|T_Z(gH(\cdot,u)+f)(x_{1},y_{1})-T_Z(gH(\cdot,u) +f)(x_{2},y_{2})\right| \\
& \leq C(p,\sigma,\Omega)||gH(\cdot,u)+f||_{p}|(x_{1},y_{1})-(x_{2},y_{2})|^{\beta}\\
& \leq C(p,\sigma,\Omega)\{||g||_{p}K+||f||_{p}\}|(x_{1},y_{1})-(x_{2},y_{2})|^{\beta}.
\end{align*}

Let $C=C(p,\sigma,\Omega)\{||g||_{p}K +||f||_{p}\}$ and $\Lambda_{M,C}$ be the set of all functions $v\in C(\overline{\Omega})$,
with $||v||_\infty\leq M$, and satisfying
\[
|v(x_{1},y_{1})-v(x_{2},y_{2})| \leq C|(x_{1},y_{1})-(x_{2},y_{2})|^{\beta}, \ \ \forall (x_{1},y_{1}), (x_{2},y_{2}) \in \overline{\Omega}.
\]
As in the proof the previous Theorem,  $\Lambda_{M,C}$ is a nonempty convex compact subset of $C(\overline{\Omega})$.
Moreover, $Pu \in \Lambda_{M,C}$, for all $u \in C(\overline{\Omega})$.

The operator $P:\Lambda_{M,C} \to \Lambda_{M,C}$ is continuous. Indeed, since $H$ is uniformly continuous on the compact set
 $U=\overline{\Omega}\times \{\zeta\in\mathbb{C} ; |\zeta| \leq M\}$, then given
$\epsilon>0$  there exists $\delta>0$ such that
\[
|H(x,y,\zeta_{1})-H(x,y,\zeta_{2})| < \displaystyle\frac{\epsilon}{M(p,\sigma,\Omega)\{||g||_{p}+1\}},
\]
for all $(x,y) \in \overline{\Omega}$ and $|\zeta_{1}-\zeta_{2}|<\delta$.
Hence, for $u,v\in \Lambda_{M,C}$ with $||u-v||_{\infty}<\delta$, we have
\begin{align*}
|Pu(x,y)-Pv(x,y)| &
\leq \left|T_Z(g\{H(\cdot,u)-H(\cdot,v)(x,y)\})\right|\\
& \leq M(p,\sigma,\Omega)||g\{H(\cdot,u)-H(\cdot,v)\}||_{p}\\
&
\leq M(p,\sigma,\Omega)||g||_{p}||H(\cdot,u)-H(\cdot,v)||_{\infty}\\
&\leq M(p,\sigma,\Omega)||g||_{p}\displaystyle\frac{\epsilon}{M(p,\sigma,\Omega)\{||g||_{p}+1\}} < \epsilon.
\end{align*}
Therefore, by Shauder Fixed Point Theorem, $P$ has a fixed point in $\Lambda_{M,C}$ that satisfies the conclusion
of the Theorem.
\end{proof}

%%%%%%%%%%%% similarity principle %%%%%%%%%%%
The classical similarity principle for generalized analytic functions
was invesitigated in {\cite{Ber-Hou-San}} and in {\cite{Mez-JDE99}}
for solutions of complex vector fields.
As a consequence of Theorem \ref{holder_continuity} and Theorem \ref{abfcont} we give here a strong version
of the similarity principle for the operator $L$:

\begin{theorem}
Let $a,b \in L^{p}(\Omega)$, $p>2+\sigma$, $\sigma>0$. Then for every $u \in L^\infty({\Omega})$ solution of equation
\begin{equation}\label{equationluaububarrasimi}
Lu=au+b\overline{u}
\end{equation}
there exists a holomorphic function $h$ defined in $Z(\Omega)$ and a function $s \in C^{\beta}({\ov{\Omega}})$ such that
\begin{equation}\label{escritauigualhzessimi}
u(x,y)=h(Z(x,y))e^{s(x,y)}, \ \  \forall (x,y)\in \Omega.
\end{equation}
Conversely, for every holomorphic function $h$ in $Z(\Omega)$ there is $s \in C^{\beta}({\overline{\Omega}})$ such that the function
$u$ given by \eqref{escritauigualhzessimi} solves \eqref{equationluaububarrasimi}.
\end{theorem}
\begin{proof}
The proof is an adaptation of that found in \cite{Mez-CM05}-Theorem 4.1. In order to keep this work as self-contained as
possible we will repeat the arguments here.

Suppose that $u \in L^\infty({\Omega})$ and that $u$ is not identically zero.
Since $L$ is smooth and elliptic in $\Omega \setminus \Sigma$ we know that $L$ is locally equivalent to a multiple of Cauchy-Riemann operator
$\partial/\partial\overline{z}$ in $\Omega \setminus\Sigma$ (see, for instance, \cite{BeCH}).
The classical similarity principle (see \cite{Bers} and \cite{Vekua}) applies and the function $u$ has the representation \eqref{escritauigualhzessimi} in the neighborhood
of each point $(x,y)\notin\Sigma$. Hence, $u$ has isolated zeros in $\Omega \setminus \Sigma$. Define the function $\phi$ in $\Omega$ by $\phi=\overline{u}/u$
at the points where $u$ is not zero and by $\phi=0$ at the points where $u=0$ and on $\Sigma$.
 Note that $\phi \in L^{\infty}({\Omega})$. It follows that $a+ b\phi \in L^{p}({\Omega})$. Consider the equation
\begin{equation}\label{ssimiloooo}
Ls=-(a+b\phi).
\end{equation}
By Theorem \ref{holder_continuity}  this equation  has a solution $s \in C^{\beta}(\overline{\Omega})$.
Define $v=ue^{s}$. A simple calculation shows that $Lv=0$. Then, $v$ can be factored as $v=h\circ Z$, with $h$ holomorphic on $Z(\Omega)$.
This proves the first part of the Theorem.

Next,
let $h$ be a holomorphic function in $Z(\Omega)$.
Define the function $\varphi$ in $Z(\Omega)$ by $\varphi=\overline{h}/h$ at the points where $h$ is not zero and by $\varphi=0$ at the points where $h=0$.
Then $\tilde{\varphi}=\varphi \circ Z \in L^{\infty}(\Omega)$.
Consequently, $b\tilde{\varphi} \in L^{p}({\Omega})$, $p>2+\sigma$.
Hence, by Theorem \ref{abfcont}, equation
\begin{center}
$Ls = a + b\tilde{\varphi} e^{\overline{s}-s}$
\end{center}
has a solution $s\in C^{\beta}({\ov{\Omega}})$. It follows at once that $u$ given by
\[
u(x,y)=h(Z(x,y)) e^{s(x,y)}, \quad (x,y)\in \Omega,
\]
solves \eqref{equationluaububarrasimi} in $\Omega$.
\end{proof}

\bibliographystyle{amsplain}

\end{document}